\documentclass[12pt,reqno]{amsart}
\usepackage{latexsym, amsmath, amscd, amssymb, amsthm,longtable} 
\usepackage[T2A]{fontenc}
\usepackage[cp1251]{inputenc}
\usepackage[english]{babel}
\newtheorem{theorem}{Theorem}
 \newtheorem*{definition}{Definition}
 
 \newtheorem{lemma}{Lemma}

\usepackage{graphicx}

\usepackage{xcolor}
 \usepackage{enumitem}
\usepackage{blkarray}
\begin{document}

\noindent

\title { First order  joint  differential projective invariants.} 

\author{Leonid Bedratyuk} 

\begin{abstract}

We present a complete algebraic description of the field of first-order joint projective invariants for configurations of \( n \) points in the plane, under the natural diagonal action of the projective group \( PGL(3,\mathbb{R}) \). For  \( n > 1 \), we construct an explicit minimal generating set for the field of absolute invariants and prove its algebraic independence. We further determine the structure of the full field of invariants as a simple algebraic extension of field of absolute invariants, generated by a single primitive relative invariant of weight~$-1$, for which we provide a closed-form expression valid for all \( n > 1 \).

\textbf{Keywords:}
projective group, invariant theory,  joint invariants, differential invariants,  computer vision.
\end{abstract}

\maketitle

\section{Introduction}

The introduction of image moment invariants by Hu in his seminal work~\cite{Hu} stands as a classical example of how invariant theory can be fruitfully applied to problems in pattern recognition. In image analysis, an \emph{image function} is a piecewise continuous real-valued function \( u(x,y) \) defined on a compact domain \( \Omega \subset \mathbb{R}^2 \), with finite, nonzero integral.

Hu studied the \emph{geometric moments} of \( u \), defined by
\[
m_{pq}(f) = \iint\limits_{\Omega} x^p y^q u(x,y)\, dx\, dy,
\]
and, employing results from the classical invariant theory of binary forms, derived image moment invariants under the action of the planar translation, scaling, and rotation groups, expressed as rational functions of the moments \( m_{pq} \). Moment invariants with respect to the full plane affine group were later obtained in~\cite{SF}. These invariants have long been used as effective tools for constructing hand-crafted features in image processing, particularly prior to the deep learning era, and they continue to find applications in various domains; see~\cite{FSB} for a comprehensive survey.

The situation becomes significantly more intricate when passing from the affine group to the plane projective group. All known attempts to extend the affine theory by defining projective analogues of geometric moments and constructing invariant expressions from them have thus far proved unsuccessful. In fact, the problem of fully characterizing projective image invariants under general projective transformations remains open; see~\cite{Open}.

The core difficulty arises from the fact that the Jacobian determinant \( J(T) \) of a projective transformation \( T \) is not constant, unlike in the affine case. If one attempts to construct projective image moments in the form
\[
\iint\limits_{\Omega} F(x, y, u(x,y))\, dx\, dy,
\]
where \( F \) is an integrable function, then under a projective change of variables, the integral transforms as
\[
\iint_{\Omega} F(T(x,y), T(u(x,y)))\, J(T)\, dx\, dy,
\]
with the Jacobian \( J(T) \) depending nontrivially on the spatial coordinates. This Jacobian factor, in general, cannot be canceled, see~\cite{Van} and \cite{Suk2000}, rendering such integrals unsuitable for constructing projective invariants.

To overcome this issue, one introduces the notion of \emph{differential projective invariants} -- functions depending not only on the variables \( x, y \) and the image intensity \( u(x,y) \), but also on its partial derivatives \( u_{i,j} \). These invariants transform predictably under the prolonged action of the projective group.
A function
\[
F(x,y,u,u_{1,0},u_{0,1}, \ldots, u_{p,q})
\]
is called a \emph{relative differential projective invariant} of order \( p+q \) and weight \( \omega \) if it satisfies the transformation law
\[
F(T(x,y), T(u), T(u_{1,0}), T(u_{0,1}), \ldots, T( u_{p,q})) = J(T)^\omega\, F(x,y,u,u_{1,0},u_{0,1}, \ldots, u_{p,q})
\]
for all projective transformations \( T \). If \( \omega = 0 \), the function is called an \emph{absolute} differential invariant.
Of particular importance are the relative invariants of weight \( \omega = -1 \), since their integrals over the domain,
\[
\iint F(x,y,u,u_{1,0},u_{0,1}, \ldots, u_{p,q}) \, dx\,dy,
\]
become projectively invariant due to the cancellation of the Jacobian factor. Thus, a natural starting point in the construction of projective differential invariants is a complete description of relative projective invariants of weight~$-1$.

Partial families of such invariants were constructed in~\cite{Open}. More recently, the work~\cite{Olver2023} employed the method of equivariant moving frames to classify projective differential invariants of images up to a given order. While these invariants are elegant and theoretically powerful, they face serious limitations in practical use.

First, real-world images are discrete, whereas the theory assumes smooth functions on continuous domains. Numerical differentiation is highly sensitive to noise and errors, especially for higher-order derivatives, making high-order differential invariants unstable and unreliable in computational practice.

Second, in the discrete setting, the group structure of projective transformations is no longer preserved. In contrast to the affine case, projective maps can cause degeneracies that make distinct discrete images nearly indistinguishable, resulting in a loss of discriminative power.

These challenges hinder the direct application of high-order differential projective invariants in image analysis and motivate the development of alternative, more robust representations compatible with discrete and noisy data.

In this work, we propose an  approach based on \emph{first-order joint projective invariants} of multiple points and their associated differential features. This framework is inherently better suited to the discrete nature of digital images.

The notion of joint invariants arises naturally in the framework of equivariant moving frames~\cite{Olver1999, Olver2001, Olver2020}. First-order joint projective invariants of weight~$-1$ for three- and four-point configurations were first introduced in~\cite{Wang}. In the present paper, we generalize these results to arbitrary configurations of~$n$ points and provide a complete algebraic solution to the problem.

The article is structured as follows. Section~1 introduces the necessary preliminaries on the projective group, its diagonal action on point configurations, and its prolongation to include first-order partial derivatives.
 We establish notation, recall the definitions of relative and absolute invariants, and review the algebraic tools used in the subsequent analysis. In Section~2, we construct the field of absolute invariants by exhibiting a minimal generating set, proving its algebraic independence, and computing the transcendence degree, which is shown to be \(2n - 8\) for \(n \ge 3\). Section~3 is devoted to the study of relative invariants of weight~$-1$. We prove that the field of relative invariants is a simple algebraic extension of the absolute field, generated by a single primitive element \(z_n\) of the weight $-1$, and we derive an explicit first-order formula for \(z_n\) in terms of point coordinates and gradient directions.


\section{Preliminaries}

This section establishes the necessary terminology and notation, and reviews the foundational aspects of the projective group action and its  prolongation to first-order derivatives of functions.

\subsection{Homogeneous coordinates and projective geometry of $\mathbb{P}^2$}

The real projective plane $\mathbb{P}^2$ is defined as the set of one-dimensional linear subspaces of~$\mathbb{R}^3$, i.e.,
\[
\mathbb{P}^2 := \bigl(\mathbb{R}^3 \setminus \{0\} \bigr) / \mathbb{R}^\times,  \mathbb{R}^\times= \mathbb{R} \setminus \{ 0\}, 
\]
where two nonzero vectors are considered equivalent if they differ by a nonzero scalar multiple:
\[
(X:Y:Z) \sim (\lambda X:\lambda Y:\lambda Z), \quad \lambda \in \mathbb{R}^\times.
\]
A point in $\mathbb{P}^2$ is thus represented by a class of homogeneous coordinates $(X:Y:Z)$, not all zero.

The plane $\mathbb{R}^2$ is embedded into the projective plane $\mathbb{P}^2$ by
\[
(x,y) \longmapsto (x:y:1),
\]
so that its image consists of the projective points whose third homogeneous coordinate is nonzero.

 A  line in~$\mathbb{P}^2$ is defined as the set of points satisfying a linear homogeneous equation:
\[
L = \left\{(X:Y:Z) \in \mathbb{P}^2 \;\middle|\; pX + qY + rZ = 0 \right\},
\]
where the triple $(p:q:r) \in \mathbb{R}^3 \setminus \{0\}$ defines the line up to a scalar multiple. 
We interpret the coefficients as a row vector \( \ell = (p, q, r) \), and say that a point \( A = (X, Y, Z)^T \in \mathbb{R}^3 \) lies on the line \( \ell \) if
\[
\ell \, A = pX + qY + rZ = 0.
\]

 The set of all lines in $\mathbb{P}^2$ is naturally identified with the set of one-dimensional subspaces of  the dual space $(\mathbb{R}^3)^*$, that is,   the vector space of all linear functionals on $\mathbb{R}^3$. This defines the \emph{dual projective plane}
\[
(\mathbb{P}^2)^* := \mathbb{P}\bigl((\mathbb{R}^3)^*\bigr) \cong \mathbb{P}^2,
\]
whose elements correspond to projective lines in the original plane.

 The fundamental duality of projective geometry is encoded in the incidence relation:
\[
\text{Point } A \in \mathbb{P}^2 \text{ lies on line } \ell \quad \Longleftrightarrow \quad  \ell \, A  = 0.
\]
Thus, points in the projective plane correspond to hyperplanes (i.e., lines) in the dual space, and vice versa. This establishes a symmetric duality between points and lines: each line in the primal plane corresponds to a point in the dual plane, and each point defines a pencil of lines — a line in the dual.

 The intersection of two lines $l_1, l_2 \in (\mathbb{R}^3)^*$ is the unique point $A \in \mathbb{P}^2$ given by the vector cross product:
\[
A = l_1 \times l_2.
\]
Conversely, the projective line passing through two distinct points $A_1, A_2 \in \mathbb{P}^2$ is defined by
\[
l = A_1 \times A_2,
\]
again using the vector cross product in $\mathbb{R}^3$.

 Three points $A_1, A_2, A_3 \in \mathbb{P}^2$ are collinear if and only if the determinant
\[
\det(A_1, A_2, A_3) = 0,
\]
i.e., their homogeneous coordinates are linearly dependent. This determinant also computes twice the signed area of the oriented triangle formed by the affine representatives of the points, when all lie in the affine chart $Z=1$.

\subsection{The plane projective group $PGL(3,\mathbb{R})$}

The group of projective transformations of the plane is given by the projective general linear group:
\[
PGL(3,\mathbb{R}) := GL(3,\mathbb{R}) / \{\lambda I \mid \lambda \in \mathbb{R}^\times\},
\]
where $GL(3,\mathbb{R})$ denotes the group of all invertible $3 \times 3$ real matrices  and $I$ is the identity matrix.
An element of $PGL(3,\mathbb{R})$ is thus an equivalence class $[A]$ of invertible matrices $A$, where $A \sim \lambda A$ for all $\lambda \in \mathbb{R}^\times$. This identification ensures that the group action is well-defined on projective points, which are also defined up to scaling.
The group $PGL(3,\mathbb{R})$ is an 8-dimensional real Lie group that  
acts naturally on the projective plane $\mathbb{P}^2$ by linear transformations in homogeneous coordinates:
\[
[X:Y:Z] \;\longmapsto\; [AX],
\]
where $A \in GL(3,\mathbb{R})$ represents the transformation. In affine coordinates $(x,y)$ (corresponding to $Z = 1$), a general projective transformation is expressed as the rational map
\begin{equation*}
T\colon(x,y)\;\longmapsto\; (\tilde x, \tilde y)=
\left(
    \frac{a_1 x + a_2 y + a_3}{c_1 x + c_2 y + c_3},
    \frac{b_1 x + b_2 y + b_3}{c_1 x + c_2 y + c_3}
\right),
\end{equation*}
where the coefficients are entries of the matrix
\[
A = 
\begin{pmatrix}
a_1 & a_2 & a_3 \\
b_1 & b_2 & b_3 \\
c_1 & c_2 & c_3
\end{pmatrix} \in GL(3,\mathbb{R}).
\]


The Jacobian determinant of the projective transformation $T$ is given by
\begin{equation*}
J(T) = \frac{D}{(c_1 x + c_2 y + c_3)^3}, \qquad  D=\det(A).
\end{equation*}
In contrast to the affine case (where $J(T)$ is constant), here the Jacobian depends on the point $(x, y)$, and the factor $(c_1 x + c_2 y + c_3)$ appears due to the rational form of the map. This nontrivial Jacobian is responsible for the complication in constructing integral invariants under projective transformations and motivates the use of differential invariants.

\subsection{Prolongation of the action to derivatives}

Let $u = u(x,y)$ be a smooth real-valued function defined on an open domain $\Omega \subset \mathbb{R}^2$. 
 We define the natural action of $T$ on the function $u$ as
\[
(T \cdot u)(x,y) = \tilde u(x,  y) := u(T^{-1}(x,y)),
\]
i.e., $\tilde u$ is obtained by precomposing $u$ with the inverse of $T$. This operation corresponds to the pullback of the function $u$ under $T$.

From this definition, one immediately obtains the  identity:
\begin{equation}\label{eq:T-inv}
\tilde u\left(\tilde x, \tilde y \right) = u(x,y).
\end{equation}
This tautological identity can be interpreted as stating that when the transformation $T$ acts simultaneously on the function and its arguments, we obtain the same function again. In other words, $u(x,y)$ is an \textit{invariant} of the transformation $T$. More meaningful examples of projective invariants arise when the action of the transformation $T$ is extended to the partial derivatives of the function.

We denote the partial derivatives of $u$ as
\[
u_{i,j} := \frac{\partial^{i+j} u(x,y)}{\partial x^i\partial y^j}, \qquad
\tilde u_{i,j} := \frac{\partial^{i+j} \tilde u( x,  y)}{\partial   x^i    y^j}.
\]

We define the field of rational functions in the coordinates and their partial derivatives up to total order $d$ as
\[
J_d := \mathbb{R}(x, y, u, u_{1,0}, u_{0,1}, \dots, u_{i,j} \mid i + j \le d).
\]

The action of the projective group $PGL(3, \mathbb{R})$ can be naturally prolonged to the  field $J_d$ by setting
\[
T  (u_{i,j}(x,y)) :=  \tilde u_{i,j}(x,y),
\]
where the derivatives $\tilde u_{i,j}$ are computed via the chain rule from the transformed function $\tilde u(x,y) = u(T^{-1}(x,y))$. This prolonged action respects the differential structure and plays a central role in the definition of differential invariants.

In this work, we restrict our attention to first-order invariants, i.e., to the case $d=1$. The transformation rules for the first-order derivatives are given explicitly in the following result.

\begin{theorem}\label{thm:first-prol}
Let $T$ be the projective transformation. Then the partial derivatives of the transformed function $\tilde u$ are given by:
\begin{align*}
\tilde u_{1,0}(\tilde x, \tilde y) &=
\frac{c_1 x + c_2 y + c_3}{D} \left(
  -\det\begin{vmatrix} b_1 & b_2 \\ c_1 & c_2 \end{vmatrix} (u_{1,0} x + u_{0,1} y)
  +\det\begin{vmatrix} b_2 & b_3 \\ c_2 & c_3 \end{vmatrix} u_{1,0}
  +\det\begin{vmatrix} b_1 & b_3 \\ c_1 & c_3 \end{vmatrix} u_{0,1}
\right), 
\\[4pt]
\tilde u_{0,1}(\tilde x, \tilde y) &=
\frac{c_1 x + c_2 y + c_3}{D} \left(
  \det\begin{vmatrix} a_1 & a_2 \\ c_1 & c_2 \end{vmatrix} (u_{1,0} x + u_{0,1} y)
  -\det\begin{vmatrix} a_2 & a_3 \\ c_2 & c_3 \end{vmatrix} u_{1,0}
  +\det\begin{vmatrix} a_1 & a_3 \\ c_1 & c_3 \end{vmatrix} u_{0,1}
\right).
\end{align*}
\end{theorem}

\begin{proof}
Differentiating the identity~(\ref{eq:T-inv}) with respect to $x$ and $y$ yields:
$$
\begin{cases}
\partial_x(\tilde u(\tilde x, \tilde y ))=u_{10}=\tilde u_{1,0}(\tilde x, \tilde y ) \partial_x(\tilde x)+\tilde u_{0,1}(\tilde x, \tilde y ) \partial_x(\tilde y),\\
\partial_y(\tilde u(\tilde x, \tilde y ))=u_{01}=\tilde u_{1,0}(\tilde x, \tilde y ) \partial_y(\tilde x )+\tilde u_{0,1}(\tilde x, \tilde y )  \partial_y(\tilde y)).
\end{cases}
$$
By solving the system, we find that 
$$
\tilde u_{1,0}=\frac{1}{J(T)}\begin{vmatrix} u_{10} &  \partial_x(T(y)) \\ u_{01} & \partial_y(T(y))\end{vmatrix}, \tilde u_{0,1}=\frac{1}{J(T)}\begin{vmatrix} \partial_x(T(x)) &  u_{10} \\ \partial_y(T(x)) & u_{01}\end{vmatrix}.
$$
After simplifications, we  obtain  the statement of the theorem.
\end{proof}

\subsection{Projective Differential Invariants}
Let
\[
F = F(x, y, u, u_{1,0}, u_{0,1}, \dots, u_{p,q} ) \in J_{p+q},
\]
be a rational function in the coordinates and derivatives. We now define the main concept of this paper.

\begin{definition}
A function $F$ is called a {(relative) projective differential invariant} of order $p+q$ and weight $\omega$ if for every projective  transformation $T $ the following identity holds:
\[
F(T(x), T(y), \tilde u, \tilde u_{1,0}, \tilde u_{0,1}, \dots, \tilde u_{p,q}) = J(T)^\omega F(x, y, u, u_{1,0}, u_{0,1}, \dots, u_{p,q}).
\]
If $\omega = 0$, then $F$ is called an \textit{absolute projective differential invariant}.
\end{definition}

Note that the weight~$\omega$ is always a rational number, since all expressions involved are constructed using rational operations.

Let $\mathcal{R}_d$ denote the field of all relative projective differential invariants of order at most $d$, i.e., functions $F \in J_d$ satisfying the invariance condition for some weight $\omega$:
\[
\mathcal{R}_d := \{ F \in J_d \mid \exists \omega \in \mathbb{R} \text{ such that } F \text{ is a projective differential invariant of weight } \omega \}.
\]
Let $\mathcal{I}_d \subset \mathcal{R}_d$ denote the subfield of all absolute differential invariants:
\[
\mathcal{I}_d := \{ F \in \mathcal{R}_d \mid \omega = 0 \}.
\]

The recent work~\cite{Olver2023} by P.~Olver provides a system of generators for the field $\mathcal{I}_d$. However, the higher-order derivatives that appear in such invariants are known to be {numerically unstable}, especially in digital image applications.

To overcome this difficulty,  we propose an alternative approach: we focus on joint projective differential invariants, which depend only on the function values and their first derivatives at a finite set of points. This construction is computationally more stable due to the use of only first-order partial derivatives. Moreover, these joint invariants admit a clear geometric interpretation in terms of gradient lines.

\subsection{Joint First-Order Projective Differential Invariants}

As discussed earlier, an effective alternative is to consider \textit{joint} projective differential invariants. These are invariants of the diagonal action of the projective group on several copies of the plane, where each copy is equipped only with first-order differential data (i.e., gradient vectors). 

Let 
$
PGL(3,\mathbb{R})
$
act diagonally on the configuration space 
$
(\mathbb{R}^{2})^n
$
by applying the same projective transformation to each point:
\begin{equation}
T: (x_i, y_i) \;\longmapsto\;
\left(
  \frac{a_1x_i + a_2y_i + a_3}{c_1x_i + c_2y_i + c_3},\;
  \frac{b_1x_i + b_2y_i + b_3}{c_1x_i + c_2y_i + c_3}
\right), \qquad i=1, \dots, n.
\end{equation}

The \textit{total Jacobian} of this transformation is given by the product of individual Jacobians at each point:
\[
J(T) = \prod_{i=1}^{n} J_i(T), \qquad
J_i(T) = \frac{D}{\bigl(c_1x_i + c_2y_i + c_3\bigr)^3}.
\]

For each point $(x_i, y_i)$, we consider a scalar function $u^{(i)} = u(x_i, y_i)$ with first-order derivatives:
\[
p_i := \frac{\partial u^{(i)}}{\partial x_i}, \qquad
q_i := \frac{\partial u^{(i)}}{\partial y_i}, \qquad i = 1,\dots,n.
\]
These quantities can be interpreted as components of the gradient vector at each point. Define the field
\[
J^{(n)}_1 := \mathbb{R}\bigl(x_1, y_1, p_1, q_1,\; \dots,\; x_n, y_n, p_n, q_n\bigr),
\]
consisting of rational functions in all $4n$ variables. The action of $PGL(3,\mathbb{R})$ on $(x_i, y_i)$ extends to $(p_i, q_i)$ via the first prolongation formulas (cf. Theorem~\ref{thm:first-prol}), and thus defines a group action on the field $J^{(n)}_1$.

Let us denote by
$\mathcal{I}_{n,0}$ the field of absolute first-order joint projective invariants, and by 
$\mathcal{I}_n$ the field of relative first-order joint projective invariants.

Both fields are subfields of $J^{(n)}_1$ and are stable under the action of $PGL(3,\mathbb{R})$. By construction, we have:
\[
\mathcal{I}_{n,0} \subset \mathcal{I}_n \subset J_1^{(n)}.
\]


It follows from Rosenlicht’s theorem on the field of invariants of algebraic group actions (cf. \cite{Ros}, \cite{Sp}) that the transcendence degree of the field $\mathcal{I}_{n,0},n \geq 3$   equals
\[
\operatorname{trdeg}_{\mathbb{R}} \mathcal{I}_{n,0} = 4n - \dim PGL(3,\mathbb{R}) = 4n - 8.
\]
This estimate holds under the assumption that the action of $PGL(3,\mathbb{R})$ on the configuration space $J_1^{(n)}$ is {generically free}, i.e., for generic configurations of points and gradient vectors, the stabilizer in $PGL(3,\mathbb{R})$ is trivial. Starting from $n \geq 3$, the diagonal action of $PGL(3,\mathbb{R})$ on $(\mathbb{R}^{2})^n$ together with generic first-order data has trivial stabilizer, and thus the conditions of Rosenlicht’s theorem are satisfied.
As a consequence, for any $n \geq 3$ the field of absolute joint first-order projective invariants is a purely transcendental extension of $\mathbb{R}$ of degree $4n - 8$.


\section{The Field of Absolute Invariants}

In the section, we explicitly construct a complete system of algebraically independent generators for the field of absolute first-order joint projective invariants~$\mathcal{I}_{n,0}$ for small values of $n = 2, 3, 4, 5, 6$, and describe the general construction for $n \geq 7$.

\subsection{Geometric Setup and Notation}

Let $A_i = (x_i, y_i, 1) \in \mathbb{R}^3$ be the homogeneous coordinates of $n$ points in general position in the projective plane~$\mathbb{P}^2$, corresponding to affine points $(x_i, y_i) \in \mathbb{R}^2$. These vectors represent points as equivalence classes under scalar multiplication:
\[
(x_i : y_i : 1) \sim \lambda (x_i : y_i : 1), \quad \lambda \in \mathbb{R}^\times.
\]

Given three such points $A_i, A_j, A_k$, define the {signed double area} of the corresponding triangle as the determinant:
\[
\delta_{ijk} := \det(A_i, A_j, A_k) =
\begin{vmatrix}
x_i & x_j & x_k \\
y_i & y_j & y_k \\
1   & 1   & 1
\end{vmatrix}.
\]

This quantity is a classical relative invariant of the weight $\frac{1}{3}$  and plays a fundamental role in the construction of joint invariants.



Now let $u^{(i)} = u(x_i, y_i)$ be a smooth scalar function defined in a neighborhood of the affine point $(x_i, y_i)$. Its gradient is denoted by
\[
\nabla u^{(i)} := \left( p_i, q_i \right) = \left( \frac{\partial u^{(i)}}{\partial x_i}, \frac{\partial u^{(i)}}{\partial y_i} \right).
\]

The affine equation of the line orthogonal to the gradient vector $\nabla u^{(i)}$ and passing through $(x_i, y_i)$ is given by:
\[
p_i (x - x_i) + q_i (y - y_i) = 0,
\quad \text{or equivalently} \quad
p_i x + q_i y + c_i = 0,
\]
{where } $c_i := -p_i x_i - q_i y_i.$

We associate to this line the homogeneous vector
\[
\ell_i := (p_i, q_i, c_i)\in \mathbb{R}^3,
\]
which represents a projective line in~$\mathbb{P}^2$. These lines will be referred to as \textit{gradient lines} associated to the function $u$ at the corresponding points.

This construction naturally provides a pair of geometric objects for each point $A_i$: a point in $\mathbb{P}^2$ and a line $\ell_i$ passing through it and orthogonal to its gradient  line. This point-line data will be the foundational ingredient in constructing joint first-order projective invariants.

\begin{lemma}\label{lem:point-line-action}
Let \(G\in GL(3,\mathbb{R})\) represent a projective transformation  
\(
T = [G].
\)
Then the action of \(T\) on projective points and lines is given by:
$$
\begin{array}{ll}
(i) &  \text{ A homogeneous point } A=(X:Y:Z)^{\mathsf{T}} \in \mathbb{R}^3\setminus\{0\} \text{ transforms by the {left} matrix }\\
    		& 		\text{multiplication  }		A \;\longmapsto\; G A.
      \\
(ii) & \text{ A homogeneous line } \ell = (p, q, r) \in (\mathbb{R}^3)^*  
      \text{ transforms by the {right} multiplication:}\\
			& 		
           \ell \;\longmapsto\; \ell\, G^{-1},
      \quad \text{or equivalently} \quad
      \ell^{\mathsf{T}} \mapsto G^{-T} \ell^{\mathsf{T}}.
    \end{array}
$$

\end{lemma}

\begin{proof}
$(i)$    
Let \(A \in \mathbb{R}^3 \setminus \{0\}\) be a column vector representing a point in the projective plane \(\mathbb{P}^2\) via the equivalence class \([A] = \{\lambda A \mid \lambda \in \mathbb{R}^{\times}\}\).  
A matrix \(G \in GL(3,\mathbb{R})\) acts on such vectors by left multiplication: \(A \mapsto G A\).
This action is compatible with projective equivalence: if \(A' = \lambda A\), then \(G A' = G (\lambda A) = \lambda (G A)\), so the image point still represents the same projective point in \(\mathbb{P}^2\).

Therefore, the action \([A] \mapsto [G A]\) is well-defined on projective points and descends to the quotient group \(PGL(3,\mathbb{R}) = GL(3,\mathbb{R})/\mathbb{R}^\times\).

\smallskip
$(ii)$  
Let \(\ell = (p, q, r)\) be a row vector (covector) representing the homogeneous linear equation \(pX + qY + rZ = 0\), which defines a line in \(\mathbb{P}^2\). This means that the point \(A \in \mathbb{P}^2\) lies on the line \(\ell\) if and only if \(\ell A = 0\).

Let us examine how the incidence relation is preserved under a projective transformation \(T = [G]\). Under this transformation, the point \(A\) becomes \(G A\). For the point \(G A\) to lie on the transformed line \(\ell'\), we must have:
\[
\ell' (G A) = 0.
\]
But since originally \(\ell A = 0\), we want to choose \(\ell'\) such that this condition remains true after the transformation. Substituting:
\[
\ell A = (\ell G^{-1}) (G A) = 0.
\]
This suggests that \(\ell'\) must be defined by:
\[
\ell' = \ell G^{-1}.
\]

That is, the correct transformation rule for lines (row vectors) is right multiplication by the inverse of \(G\): \(\ell \mapsto \ell G^{-1}\). This ensures that incidence is preserved:
\[
\ell A = 0 \quad \Longleftrightarrow \quad \ell G^{-1} (G A) = 0 \quad \Longleftrightarrow \quad \ell' (G A) = 0.
\]

{Equivalently}, if we write the covector \(\ell\) as a column vector \(\ell^{\mathsf T}\), then the action becomes:
\[
\ell^{\mathsf T} \mapsto G^{-T} \ell^{\mathsf T}.
\]

This follows from the identity:
\[
(\ell G^{-1})^{\mathsf T} = (G^{-1})^{\mathsf T} \ell^{\mathsf T} = G^{-T} \ell^{\mathsf T}.
\]

Hence, both forms \(\ell \mapsto \ell G^{-1}\) (row vector) and \(\ell^{\mathsf T} \mapsto G^{-T} \ell^{\mathsf T}\) (column covector) describe the same transformation of the line under \(T\), and incidence is preserved.

\end{proof}

\subsection{The Case \(n=3\)}
We have  $
\operatorname{trdeg} \mathcal I_{3,0}=4\cdot3-8=4.
$
We now describe the field of absolute first-order joint projective invariants in the case of three points in general position.

\begin{theorem}
\label{thm:n3}
The field of absolute first-order joint projective invariants is the rational function field
\[
\mathcal{I}_{3,0}
\;=\;
\mathbb{R}\bigl(\zeta_{12},\zeta_{23},\zeta_{13},\tau\bigr),
\]
where
\[
\begin{aligned}
\zeta_{12} &:= \det(\ell_1, \ell_2, L_3), &
\zeta_{23} &:= \det(L_1, \ell_2, \ell_3), \\[4pt]
\zeta_{13} &:= \det(\ell_1, L_2, \ell_3), &
\tau    &:= \det(\ell_1, \ell_2, \ell_3)\cdot \delta_{123}.
\end{aligned}
\]
 are algebraically independent   absolute first-order joint projective invariants. 
\end{theorem}

\begin{proof}
Let \(T \in PGL(3,\mathbb{R})\) be represented by a matrix \(G \in GL(3,\mathbb{R})\).  Consider, for example,
\[
\zeta_{12} = \det(\ell_1, \ell_2, L_3),
\]
where \(L_3 = A_1 \times A_2\). 

By Lemma~\ref{lem:point-line-action}, item $(ii)$, each gradient line transforms as \(\ell_j \mapsto G^{-T} \ell_j\). 
Furthermore, using the standard identity for the cross product of transformed points and item (i) of Lemma~\ref{lem:point-line-action}, we have:
\[
L_3 = A_1 \times A_2 \;\mapsto\; (G A_1) \times (G A_2) = (\det G)\, G^{-T} L_3.
\]
Therefore,
\[
\det(G^{-T}\ell_1, G^{-T}\ell_2, (\det G)\,G^{-T}L_3)
= (\det G^{-T}) (\det G)\, \det(\ell_1, \ell_2, L_3)
= \zeta_1,
\]
since \((\det G^{-T}) = (\det G)^{-1}\), and the total product becomes 1. Analogous computations confirm the invariance of \(\zeta_2\) and \(\zeta_3\).

 Note that
\[
\tau = \det(\ell_1, \ell_2, \ell_3) \cdot \det(A_1, A_2, A_3).
\]
Under \(T\), these quantities transform as follows, by items $(ii)$ and $(i)$ of Lemma~\ref{lem:point-line-action}, respectively:
\begin{align*}
&\det(\ell_1, \ell_2, \ell_3) \mapsto \frac{(c_1 x_1+c_2 y_1 +c_3)(c_1 x_2+c_2 y_2 +c_3)(c_1 x_3+c_2 y_3 +c_3)}{\det G} \det(\ell_1, \ell_2, \ell_3),\\
&\det(A_1, A_2, A_3) \mapsto \frac{\det G}{(c_1 x_1+c_2 y_1 +c_3)(c_1 x_2+c_2 y_2 +c_3)(c_1 x_3+c_2 y_3 +c_3)}\, \det(A_1, A_2, A_3),
\end{align*}
so the product \(\tau\) remains invariant.  We see  that each factor of the product
is a relative invariant of weight~\(-\tfrac{1}{3}\)  or \(\tfrac{1}{3}\).

 The explicit form of the invariants,  up to sign,  is as follows:
\[
\begin{aligned}
\zeta_{12} &=
 \bigl(p_1(x_1-x_2)+q_1(y_1-y_2)\bigr)
 \bigl(p_2(x_1-x_2)+q_2(y_1-y_2)\bigr),\\[6pt]
\zeta_{23} &=
 \bigl(p_2(x_2-x_3)+q_2(y_2-y_3)\bigr)
 \bigl(p_3(x_2-x_3)+q_3(y_2-y_3)\bigr),\\[6pt]
\zeta_{13} &=
 \bigl(p_3(x_1-x_3)+q_3(y_1-y_3)\bigr)
 \bigl(p_1(x_1-x_3)+q_1(y_1-y_3)\bigr),\\[6pt]
\end{aligned}
\]
$$
\tau= 
\begin{vmatrix}
 p_1 & q_1 & p_1x_1+q_1y_1\\
 p_2 & q_2 & p_2x_2+q_2y_2\\
 p_3 & q_3 & p_3x_3+q_3y_3
\end{vmatrix}\cdot \begin{vmatrix}
 x_1 & y_1 & 1\\
 x_2 & y_2 & 1\\
 x_3 & y_3 & 1
\end{vmatrix}.
$$
Direct computation shows that the Jacobian matrix of these four functions  has the rank 4. This implies that the invariants $\zeta_{12}, \zeta_{23}, \zeta_{13}, \tau$ are algebraically independent. 

\end{proof}

Geometrically, the triple determinant
\[
\zeta_{12} = \det(\ell_1, \ell_2, L_3)
\]
represents the oriented area of the triangle in the dual projective plane~$(\mathbb{P}^2)^*$, whose vertices correspond to the lines~$\ell_1$, $\ell_2$, and~$L_3$. These lines — the two gradient directions~$\ell_1$, $\ell_2$ and the side~$A_1A_2$ of the reference triangle (represented by~$L_3$) — define three distinct points in the dual space. The determinant~$\zeta_{12}$ measures the signed area of the triangle they span and thus encodes their relative position.

In particular, $\zeta_{12}$ vanishes precisely when the lines~$\ell_1$, $\ell_2$, and~$L_3$ are concurrent — that is, when the two gradient lines intersect on the side~$A_1A_2$. This concurrency condition is invariant under projective transformations, and~$\zeta_{12}$ captures it as a numerical invariant.

The same interpretation applies to the cyclic analogues
$\zeta_{23}$ and $\zeta_{13}$.

The invariant~$\tau$ simultaneously detects the non-degeneracy of two geometric configurations: the point triangle~$A_1A_2A_3$ in the projective plane~$\mathbb{P}^2$, and the gradient triangle formed by the lines~$\ell_1, \ell_2, \ell_3$ in the dual plane~$\mathbb{P}^{2*}$. It takes a nonzero value if and only if the points~$A_1, A_2, A_3$ are not collinear and the lines~$\ell_1, \ell_2, \ell_3$ are not concurrent or pairwise parallel.

\subsection{\textbf{The case \boldmath{$n=4$}}}

For four points in general position   we have 
$
\operatorname{trdeg} I_{4,0}= 8.
$
As in the case \(n=3\), we include the three basic invariants
\(\zeta_{12}, \zeta_{13}, \zeta_{23}\)
and the mixed invariant
\(\tau = \det(l_1, l_2, l_3) \cdot \det(A_1, A_2, A_3)\).
To complete the set of eight generators, we introduce three new invariants
\[
\zeta_{14},\quad \zeta_{24},\quad \zeta_{34}
\]
defined analogously as signed volumes of appropriate triangles in the dual projective plane.

In addition, we construct a fourth-order mixed invariant that generalizes \(\tau\) to four points:
\[
\sigma = 
\delta_{123} \, \delta_{124} \, \delta_{134} \, \delta_{234} \cdot
\det(l_1, l_2, l_3) \det(l_1, l_2, l_4)
\cdot \det(l_1, l_3, l_4) \det(l_2, l_3, l_4).
\]
The invariance of \(\sigma\) under projective transformations is established in the same way as in the case \(n=3\), by using the transformation rules from Lemma~\ref{lem:point-line-action}.

This leads to the following description of the field of invariants:
\begin{theorem}
The field of absolute first-order joint projective invariants is the purely transcendental field
\[
I_{4,0} \;=\;
\mathbb{R}\bigl(
  \zeta_{12},\zeta_{13},\zeta_{14},\zeta_{23},\zeta_{24},\zeta_{34},\;
  \tau,\;\sigma
\bigr).
\]
The eight invariants listed are algebraically independent, and form a transcendence basis of the field \(\mathcal I_{4,0}\).
\end{theorem}

The algebraic independence of these invariants follows from direct computation.

\subsection{The cases \boldmath{$n=5$} and \boldmath{$n=6$}}

For $n=5$ and $n=6$, the construction of absolute joint projective invariants follows the same principles as in the case $n=4$: we combine determinants of triples of gradient lines (dual points in the dual plane) with determinants of point triples (primal areas), forming mixed expressions that are invariant under the diagonal action of the projective group .
In total, there are 12 algebraically independent absolute first-order invariants for \(n = 5\), and 16 for \(n = 6\).

We omit the detailed constructions and state only the final result.

\begin{theorem}
The structure of the invariant field in the cases \(n = 5\) and \(n = 6\) is described as follows:
\begin{enumerate}\itemsep=1em

\item[\textnormal{(i)}] For $n=5$, the field of absolute first-order joint projective invariants is the purely transcendental field
\[
I_{5,0}
 \;=\;
 \mathbb{R}\!\bigl(
   \zeta_{12},\zeta_{13},\zeta_{14},\zeta_{15},\;
   \zeta_{23},\zeta_{24},\zeta_{25},\;
   \zeta_{34},\zeta_{35},\zeta_{45};\;
   \tau,\sigma
 \bigr),
\]
where the twelve listed invariants are algebraically independent.

\item[\textnormal{(ii)}] For $n=6$, the field of absolute first-order joint projective invariants is the purely transcendental field
\[
I_{6,0}
 \;=\;
 \mathbb{R}\!\bigl(
   \zeta_{12},\zeta_{13},\zeta_{14},\zeta_{15},\zeta_{16},\;
   \zeta_{23},\zeta_{24},\zeta_{25},\zeta_{26},\;
   \zeta_{34},\zeta_{35},\zeta_{36},\;
   \zeta_{45},\zeta_{46},\zeta_{56};\;
   \tau
 \bigr),
\]
where the sixteen listed invariants are algebraically independent.
\end{enumerate}
\end{theorem}

The algebraic independence of the listed invariants can be verified by symbolic computation in any standard computer algebra system.

\subsection{The case \boldmath{$n>6$}}
To construct a transcendence basis for the field of absolute joint projective invariants of \(n\) points in general position, we proceed inductively, following the pattern observed in the cases \(n \leq 6\). We introduce the following families of invariants:

\begin{align*}
&Z_2 = \{\zeta_{12}\},\\
&Z_3 := \{\zeta_{12},\; \zeta_{23}\},\\
&Z_4 := \{\zeta_{14},\; \zeta_{24},\; \zeta_{34}\},\\
&Z_k := \{\zeta_{1k},\; \zeta_{2k},\; \zeta_{3k},\; \zeta_{4k}\}, \qquad k = 5, \dots, n-1,\\
&Z_n := \{\zeta_{1n},\; \zeta_{2n},\; \zeta_{3n},\; \zeta_{4n},\; \zeta_{5n},\; \zeta_{6n}\}.\\
\end{align*}

Counting the invariants we obtain
\[
|Z_2|+|Z_3|+|Z_4|+\sum_{k=5}^{n-1}|Z_k|+|Z_n|
        = 1+2+3+4(n-5)+6 = 4n-8.
\]
Denote
\(
\mathcal G_n := \bigcup_{k=2}^{n} Z_k.
\)

\begin{theorem}
Let $n \geq 7$. Then the set $\mathcal G_n$ of $4n-8$ invariants is algebraically independent and forms a transcendence basis for the field $\mathcal I_{n,0}$:
\[
\mathcal I_{n,0} = \mathbb{R}\bigl( \mathcal G_n \bigr).
\]

\end{theorem}

\begin{proof}
We present the proof in detail for the case \(n = 7\); the inductive step to arbitrary \(n\) then follows directly.

The idea of the proof is to explicitly construct a minor of the Jacobian matrix of the invariants $\mathcal{G}_n$ that attains the maximal rank $20$ for a particular choice of variables.

Before computing the Jacobian, we normalize six parameters by fixing the coordinates of the first three points:
\[
x_1=y_1=0,\;
x_2=1,\;y_2=0,\;
x_3=0,\;y_3=1,\;
\]
Note that this normalization does not increase the Jacobian rank, so any maximal rank established after normalization remains valid.
 The pair of gradient variables \(p_1, q_1\) are treated as parameters and are excluded from differentiation.

The residual variables, i.e., the 20 coordinates that remain after the initial normalization, are those with respect to which we compute the Jacobian. These variables are grouped block-wise as follows:
\[
(p_2)\;;\;
(p_3,q_3)\;;\;
(y_4,p_4,q_4)\;;\;
(x_5,y_5,p_5,q_5)\;;\;
(x_6,y_6,p_6,q_6)\;;\;
(x_7,y_7,p_7,q_7,q_2,y_4).
\]

Denote by \(J\) the Jacobian of the \(20\) invariants
\[
\mathcal G_7=
\bigl\{
\zeta_{12}\;;\;
\zeta_{13},\zeta_{23}\;;\;
\zeta_{14},\zeta_{24},\zeta_{34}\;;\;
\zeta_{15},\zeta_{25},\zeta_{35},\zeta_{45}\;;\;
\zeta_{16},\zeta_{26},\zeta_{36},\zeta_{46}\;;\;
\zeta_{17},\zeta_{27},\zeta_{37},\zeta_{47},\zeta_{57},\zeta_{67}
\bigr\}.
\]
With the column order  the $20\times 20$ matrix \(J\) is   almost
block lower–triangular:

\[
\arraycolsep4pt
J=
\begin{blockarray}{c@{\;}c@{\;}c@{\;}c@{\;}c@{\;}c}
\begin{block}{(c@{\;}c@{\;}c@{\;}c@{\;}c@{\;}c)}
\boxed{D_2} & 0 & 0 & 0 & 0 & * \\                 
* & \boxed{D_3} & 0 & 0 & 0 & * \\                
* & * & \boxed{D_4} & 0 & 0 & * \\                
* & * & * & \boxed{D_5} & 0 & * \\                
* & * & * & * & \boxed{D_6} & * \\                
* & * & * & * & * & \boxed{D_7} \\                
\end{block}
\end{blockarray}
\]

\smallskip\noindent
Here the diagonal blocks have the following sizes:
\[
\dim D_2=1\times 2,\;
\dim D_3=2\times 2,\;
\dim D_4=3\times 4,\;
\dim D_5=\dim D_6=4\times4,\;
\dim D_7=6\times 6.
\]
The nonzero entries above \(D_7\) are due to the placement of \(q_2\) and \(y_4\) at the end of the ordering; these variables can appear in derivatives of invariants \(\zeta_{i,j}\) with indices 2 and 4.

To prove that the Jacobian \(J\) has full rank 20, we adopt the following strategy. Since \(J\) depends on 20 variables, it suffices to exhibit a specific choice of values (a specialization) for which the rank of \(J\) is exactly 20. We will construct such a specialization in which all diagonal blocks \(D_i\), for \(i < 7\), are nondegenerate (i.e., have nonzero determinants), and all entries above the block \(D_7\) vanish. In this case, the Jacobian matrix takes a block-diagonal form with full-rank blocks on the diagonal, which guarantees that the rank will be 20 if such a specialization exists.

For each block of residual variables, we consider a minor of the Jacobian matrix \(J\), formed by selecting the columns corresponding to the variables in that block, along with two additional columns associated with the variables \(q_2\) and \(q_4\).

For the block \(D_2\), after the initial normalization we have:
\[
\zeta_{1,2} = p_1 p_2.
\]
Therefore, the Jacobian submatrix \(J\) corresponding to the columns \(p_2, q_2, q_4\) and the row \(\zeta_{1,2}\) takes the form \([p_1,\; 0,\; 0]\), which implies that \(|D_2| = p_1\), and the two remaining entries in the row are zero.

For the block \(D_3\), we consider the \(2 \times 4\) minor defined by the columns \(p_3, q_3, q_2, y_4\) and the rows \(\zeta_{1,3}, \zeta_{2,3}\), which, after normalization, take the form:
\begin{gather*}
\zeta_{1,3} = q_1 q_3,\\
\zeta_{2,3} = (p_3 - q_3)(p_2 - q_2).
\end{gather*}

Direct computation shows that the corresponding submatrix of the Jacobian has the form:
\[
\left(
\begin{array}{cccc}
0 & 1 & 0 & 0 \\
p_2 - q_2 &  q_2-p_2  &  q_3-p_3  & 0
\end{array}
\right).
\]

Under the specialization \(p_3 = q_3\), the last two columns become zero, and the determinant of the \(2 \times 2\) submatrix formed by the first two columns becomes:
\[
|D_3| =
\begin{vmatrix}
0 & 1 \\
p_2 - q_2 & q_2-p_2 
\end{vmatrix}
= p_2 - q_2.
\]

For the block \(D_4\), we consider the \(3 \times 5\) minor defined by the columns \(x_4, y_4, p_4, q_2, y_4\) and the normalized rows \(\zeta_{1,4}, \zeta_{2,4}, \zeta_{3,4}\), which take the form:
\begin{gather*}
\zeta_{{1,4}}= \left( -p_{{4}}x_{{4}}-q_{{4}}y_{{4}} \right)  \left( -p_{{1}}x_{{4}}-q_{{1}}y_{{4}} \right),\\
\zeta_{{2,4}}= \left( -p_{{4}}
x_{{4}}-q_{{4}}y_{{4}}+p_{{4}} \right)  \left( -p_{{2}}x_{{4}}-q_{{2}}
y_{{4}}+p_{{2}} \right),\\
\zeta_{{3,4}}= \left( -p_{{4}}x_{{4}}-q_{{4}}
y_{{4}}+q_{{4}} \right)  \left( -p_{{3}}x_{{4}}-q_{{3}}y_{{4}}+q_{{3}}
 \right). 
\end{gather*}
By direct computation, one can verify that under the specialization 
$$
 p_{{1}}={\frac { \left( x_{{4}}-2 \right) q_{{1}}}{x_{{4}}}},p_{{2}}={\frac { \left( x_{{4}}-2 \right) q_{{2}}}{x_{{4}}}},p_{{3}}=q
_{{3}},p_{{4}}=-{\frac {q_{{4}} \left( x_{{4}}-2 \right) }{x_{{4}}}},y_{{4}}=0
$$
the corresponding minor of the Jacobian matrix takes the form:
$$
\left( \begin {array}{ccccc} x_{{4}} \left( x_{{4}}-2 \right) q_{{1}}&0&-2\,{\frac {q_{{4}} \left( x_{{4}}-2 \right) ^{2}q_{{1}}}{x_{{4}}}}
&0&0\\ \noalign{\medskip}{\frac { \left( x_{{4}}-1 \right) ^{2}
 \left( x_{{4}}-2 \right) q_{{2}}}{x_{{4}}}}&0&-2\,{\frac {q_{{4}}
 \left( x_{{4}}-2 \right) ^{2}q_{{2}} \left( x_{{4}}-1 \right) }{{x_{{
4}}}^{2}}}&0&0\\ \noalign{\medskip}x_{{4}}q_{{3}} \left( x_{{4}}-1
 \right) &-q_{{3}} \left( x_{{4}}-1 \right) &-2\,{\frac {q_{{4}}
 \left( x_{{4}}-1 \right) ^{2}q_{{3}}}{x_{{4}}}}&0&0\end {array}\right).
$$
Then 
$$
|D_4|= \left|\begin {array}{ccc} x_{{4}} \left( x_{{4}}-2 \right) q_{{1}}&0&{\frac {q_{{4}} \left( x_{{4}}-2 \right) ^{2}q_{{1}}}{x_{{4}}}}
\\ \noalign{\medskip}{\frac { \left( x_{{4}}-1 \right) ^{2} \left( x_{
{4}}-2 \right) q_{{2}}}{x_{{4}}}}&0&{\frac {q_{{4}} \left( x_{{4}}
-2 \right) ^{2}q_{{2}} \left( x_{{4}}-1 \right) }{{x_{{4}}}^{2}}}
\\ \noalign{\medskip}x_{{4}}q_{{3}} \left( x_{{4}}-1 \right) &-q_{{3}}
 \left( x_{{4}}-1 \right) &{\frac {q_{{4}} \left( x_{{4}}-1
 \right) ^{2}q_{{3}}}{x_{{4}}}}\end {array} \right|={\frac { \left( x_{{4}}-2 \right) ^{3}q_{{1}}q_{{4}}q_{{2}}
 \left( x_{{4}}-1 \right) ^{2}q_{{3}}}{{x_{{4}}}^{2}}}
.
$$
For the block \(D_5\), we apply the following specialization:
$$
p_{{5}}={\frac {q_{{5}} \left( x_{{4}}-2 \right) }{x_{{4}}}},y_{{5}
}=0 
$$
We omit the explicit form of the minor and the intermediate steps of the determinant computation, and present only the final result. Combined with the previous specializations, this eliminates the two nonzero columns at the end of the block, yielding:
\[
|D_5| = -4\,\frac{q_4 q_5^2 (x_4 - 2)^4 (x_4 - x_5)\, q_1 x_5 (x_5 - 1)^2 q_2 q_3}{x_4^4}.
\]

The structure of block \(D_6\) is analogous to that of \(D_5\). Therefore, the additional specialization
$$
p_{{6}}={\frac {q_{{6}} \left( x_{{4}}-2 \right) }{x_{{4}}}},y_{{6}}=0 
$$
also eliminates the two nonzero columns at the end of the block. As a result, we obtain:
$$
|D_6|=-4\,{\frac {q_{{4}}{q_{{6}}}^{2} \left( x_{{4}}-2 \right) ^{4} \left( 
x_{{4}}-x_{{6}} \right) q_{{1}}x_{{6}} \left( x_{{6}}-1 \right) ^{2}q_
{{2}}q_{{3}}}{{x_{{4}}}^{4}}}.
$$

Finally, we turn to block \(D_7\). The previous specializations have eliminated all nonzero entries above this block in the last six columns. Taking into account all prior specializations, a direct computation yields:
\begin{gather*}
|D_7|={{x_{
{4}}}^{-4}} q_{{1}}q_{{3}}q_{{4}}q_{{5}}q_{{6}}x_{{5}}x_{{6}}{y_{{7}}}^{2}
 \left( x_{{4}}-2 \right)  \left( p_{{7}}x_{{7}}+q_{{7}}y_{{7}}-p_{{7}
} \right)  \left( p_{{7}}x_{{4}}-q_{{7}}x_{{4}}+2\,q_{{7}} \right) \times \\ \times 
 \left( p_{{7}}x_{{4}}x_{{7}}+p_{{7}}y_{{7}}x_{{4}}+2\,p_{{7}}x_{{7}}y
_{{7}}-q_{{7}}x_{{7}}x_{{4}}-q_{{7}}y_{{7}}x_{{4}}+2\,q_{{7}}{y_{{7}}}
^{2}-p_{{7}}x_{{4}}+q_{{7}}x_{{4}}+2\,q_{{7}}x_{{7}}-2\,q_{{7}}
 \right) \times \\ \times \left( p_{{7}}{x_{{4}}}^{2}-p_{{7}}x_{{4}}x_{{7}}-q_{{7}}{x_
{{4}}}^{2}+q_{{7}}x_{{7}}x_{{4}}-2\,q_{{7}}y_{{7}}x_{{4}}+2\,q_{{7}}x_
{{4}}-2\,q_{{7}}x_{{7}} \right)  \left( x_{{5}}-x_{{6}} \right). 
\end{gather*}

Having obtained explicit expressions for all determinants \(D_i\), it is now straightforward to choose values of the variables such that all these determinants are nonzero. For instance, setting
\[
x_4 = 4,\quad x_5 = 5,\quad x_6 = 6,\quad x_7 = 7,\quad p_7 = 0,\quad q_7 = 1,
\]
and assigning arbitrary nonzero values to all other involved variables (with the additional condition \(p_2 \neq q_2\)), we ensure that each determinant \(D_i\) is nonvanishing.

Thus, we have demonstrated the existence of a specialization of the variables that allows us to extract a \(20 \times 20\) square minor of the Jacobian matrix \(J\), which has a block-diagonal structure with nonzero diagonal determinants. Hence, the Jacobian \(J\) attains the maximal rank of 20.

The proof for \(n > 7\) proceeds in the same manner. For each new block of variables \((x_n, y_n, p_n, q_n)\), we assign the values \(x_n = n\), \(y_n = 1\), \(p_n = 0\), and \(q_n = 1\), while the parameters from previous steps are set as before -- in particular, \(x_{n-1} = n-1\), \(y_n = 0\), and \(p_{n-1} = \frac{1}{2} q_{n-1}\). With this construction, the Jacobian matrix has rank \(4n - 8\) by design.

Hence, we conclude that \(\mathcal{G}_n\) indeed forms a minimal transcendence basis for the field of invariants.
\end{proof}

Thus, we have obtained a complete and explicit description of all algebraically independent first-order absolute joint projective invariants for any number of points \( n \geq 3 \), providing a minimal generating set for the corresponding field of invariants.

\subsection{What happens in the case \boldmath{$n=2$}?}

At first glance, one might expect that no absolute first-order invariants exist for $n=2$. Indeed, the dimension count gives:
\[
\dim J^{(2)}_1 = 4 \cdot 2 = 8, \qquad \dim PGL(3,\mathbb{R}) = 8,
\]
and so the general formula for the transcendence degree
yields zero in this case.

However, the expression
$
\zeta_{12}
$
provides a counterexample, as it is straightforward to verify that $\zeta_{12}$ is invariant under the action of $PGL(3,\mathbb{R})$.

The apparent contradiction is resolved by recognizing that the formula $4n - 8$ assumes a \emph{generic} configuration with a trivial stabilizer under the action of $PGL(3,\mathbb{R})$. However, for $n=2$ this assumption fails.
Indeed, the two points $A_1,A_2$ determine the line $L:=A_1A_2$, and the two gradient lines $\ell_1,\ell_2$ meet at a third point $P=\ell_1\cap\ell_2$.
Any projective transformation that fixes the three points $A_1,A_2,P$ acts as a one–parameter homology with axis $L$ and centre $P$; it preserves both points and their associated gradient directions.
Consequently
\[
\dim\operatorname{Stab}_{PGL(3)}(J^{(2)}_1)=1,
\qquad
\operatorname{trdeg}_{\mathbb R}\mathcal I_{2,0}=8-(8-1)=1,
\]
so one absolute first–order invariant does exist.

\section{Field of Relative Invariants}

In this section, we investigate the structure of the field $\mathcal{I}_n$ of relative first-order joint projective invariants under the diagonal action of the projective group $PGL(3, \mathbb{R})$. Our goal is to describe $\mathcal{I}_n$ explicitly as an algebraic extension of $\mathcal{I}_{n,0}$. We will show that this extension is always simple: there exists a single primitive relative invariant $z$ such that every other relative invariant is a rational function in $z$ with coefficients in $\mathcal{I}_{n,0}$. The key to this result lies in the analysis of the weight group of relative invariants and the minimality of the weight exponent.

The intuition behind the simplicity of the extension is as follows: suppose $z$ is a relative invariant with minimal nonzero weight $\operatorname{wt}(z)$ (in absolute value). Then for any other relative invariant $z'$, the ratio
\[
\frac{z'}{z^{\frac{\operatorname{wt}(z')}{\operatorname{wt}(z)}}}
\]
is an absolute invariant, and thus lies in $\mathcal{I}_{n,0}$. Therefore, $z' \in \mathcal{I}_{n,0}(z)$, and the extension is simple. However, to turn this argument into a rigorous proof within the field of rational functions, we must ensure that the quantity $\frac{\operatorname{wt}(z')}{\operatorname{wt}(z)}$ is an integer for all $z'$. This requires, for each $n$, the identification of a primitive generator $z_n$ whose weight divides the weight of every other relative invariant. 

\subsection{Construction of Relative Invariants}

For each \( n > 1 \), we seek relative invariants — candidates for primitive generators of the extension — in the form of products of determinants
\[
\Delta_{ijk}=\det(\ell_1, \ell_2, \ell_3),
\]
which transform under the action of the projective group as follows:
\[
\Delta_{ijk} \;\longmapsto\; \frac{(c_1x_i + c_2y_i + c_3)
(c_1x_j + c_2y_j + c_3)(c_1x_k + c_2y_k + c_3)}{D}
\cdot \Delta_{ijk}.
\]
Recall that the total Jacobian of the group action \( T \) is given by
\[
J(T) = \prod_{i=1}^n J_i(T) = \frac{D^n}{\displaystyle\prod_{i=1}^n
(c_1x_i + c_2y_i + c_3)^3}.
\]

Let \(\mathcal{S} := \bigl\{ S \subset \{1, \dots, n\} \mid |S| = 3 \bigr\}\) denote the set of all 3-element subsets of indices. For each \( S \in \mathcal{S} \), let \( m_S \in \mathbb{Z} \) be an integer exponent, and define the product
\[
\Delta_n := \prod_{S \in \mathcal{S}} \Delta_S^{m_S}.
\]

Under a projective transformation \(T \in PGL(3,\mathbb{R})\), each determinant \(\Delta_S\) transforms as
\[
\Delta_S \mapsto D^{-1} \cdot \prod_{s \in S} (c_1 x_s + c_2 y_s + c_3) \cdot \Delta_S,
\]
so the entire product transforms according to
\[
\Delta_n \mapsto \left( \prod_{i=1}^n (c_1 x_i + c_2 y_i + c_3)^{e_i} \right)
\cdot D^{-\sum_S m_S} \cdot \Delta_n,
\]
where
\[
e_i := \sum_{\substack{S \in \mathcal{S} \\ i \in S}} m_S.
\]

In order for \(\Delta_n\) to be a relative invariant of weight \(\omega\), the exponents \(m_S\) must satisfy the following system of linear equations:
\begin{equation}\label{eq:weights}
\begin{cases}
\displaystyle\sum_{\substack{S \in \mathcal{S} \\ i \in S}} m_S = -3\omega, & \text{for } i = 1, \dots, n, \\
\displaystyle\sum_{S \in \mathcal{S}} m_S = -n\omega. &
\end{cases}
\end{equation}

To proceed, we perform direct computations and determine the explicit form of \(\Delta_n\) for small values of \(n\).

\subsection{ Explicit Form of \(\Delta_n\) for \(n \leq 6\)}

For \(n = 3\), there is a single triple: \(S = \{1, 2, 3\}\). In this case,
\[
\Delta = \Delta_{123}^m, \qquad
\omega = -\frac{m}{3}, \quad m \in \mathbb{Z},
\]
so the minimal absolute value of the weight is  \(|\omega| = \tfrac{1}{3}\) when \(m = 1\). The associated relative invariant is given by
\[
z_3 := \Delta_{123}, \qquad \operatorname{wt}(z_3) = -\tfrac{1}{3}.
\]

For \(n = 4\), there are four possible index triples:
\[
\mathcal{S} = \{\{1,2,3\},\, \{1,2,4\},\, \{1,3,4\},\, \{2,3,4\}\},
\]
and the corresponding system of equations for the exponents \(m_S\) takes the form:
$$
\begin{cases}
m_{{1}}+m_{{2}}+m_{{3}}=-3\,\omega, \\
m_{{1}}+m_{{2}}+m_{{4}}=-3\,\omega, \\
m_{{1}}+m_{{3}}+m_{{4}}=-3\,\omega, \\
m_{{2}}+m_{{3}}+m_{{4}}=-3\,\omega,\\
m_{{1}}+m_{{2}}+m_{{3}}+m_{{4}}=-4\,\omega. 
\end{cases}
$$
The  solution to this system is
$$
\left\{ m_{{1}}=-\omega,m_{{2}}=-\omega,m_{{3}}=-\omega,m_{{4}}=-\omega \right\} 
$$
Hence, the minimal absolute weight is \(1\), and the corresponding relative invariant is given by
$$
z_4=\Delta_{{1,2,3}}\Delta_{{1,2,4}}\Delta_{{1,3,4}}\Delta_{{2,3,4}}.
$$

For  $n=5$,    we obtain the following system with 10 unknowns:
$$
\begin{cases}
 m_{{1}}+m_{{2}}+m_{{3}}+m_{{4}}+m_{{5}}+m_{{6}}=-3\,\omega,\\
m_{{1}}+m_{{2}}+m_{{3}}+m_{{7}}+m_{{8}}+m_{{9}}=-3\,\omega, \\
m_{{1}}+m_{{4
}}+m_{{5}}+m_{{7}}+m_{{8}}+m_{{10}}=-3\,\omega,\\ m_{{2}}+m_{{4}}+m_{{6}}
+m_{{7}}+m_{{9}}+m_{{10}}=-3\,\omega,\\ m_{{3}}+m_{{5}}+m_{{6}}+m_{{8}}+m
_{{9}}+m_{{10}}=-3\,\omega,\\m_{{1}}+m_{{2}}+m_{{3}}+m_{{4}}+m_{{5}}+m_{
{6}}+m_{{7}}+m_{{8}}+m_{{9}}+m_{{10}}=-5\,\omega 
\end{cases}
$$
Let us solve it and obtain the following:
$$
\begin{cases}
 m_{{1}}=m_{{10}}+\omega+m_{{9}}+m_{{6}},\\
m_{{2}}=m_{{5}}+m_{{8
}}+m_{{10}}+\omega, \\
m_{{3}}=-m_{{9}}-m_{{5}}-m_{{6}}-m_{{8}}-m_{{10}}-3
\,\omega,\\
m_{{4}}=-m_{{5}}-m_{{6}}-2\,\omega-m_{{10}},\\
m_{{7}}=-2\,\omega-m_{{8}}-m_{{9}}-m_{{10}}
\end{cases}
$$
By setting all free variables to zero, we obtain the following particular solution:
$$
\begin{cases}
 m_{{1}}=\omega,\\
m_{{2}}=\omega, \\
m_{{3}}=-3
\,\omega,\\
m_{{4}}=-2\,\omega,\\
m_{{7}}=-2\,\omega
\end{cases}
$$
The smallest value of $\omega$  (in absolute value) that yields integer solutions is $-1$, which gives the following relative invariant:
$$
z_5={\frac {{\Delta_{{1,2,5}}}^{3}{\Delta_{{1,3,4}}}^{2}{\Delta_{{2,3,4}}}
^{2}}{\Delta_{{1,2,3}}\Delta_{{1,2,4}}}},\operatorname{wt}(z_5) = -1.
$$
Note that assigning nonzero values to the free variables would only multiply $z_5$  by some absolute invariant, and cannot decrease its weight.

For $n=6$, we omit the calculations and present only the final form of the relative invariant.
\[
z_6 := \frac{ \Delta_{125} \Delta_{126} \Delta_{134} \Delta_{234} }
             { \Delta_{123} \Delta_{124} }, \qquad \operatorname{wt}(z_6) = -\tfrac13.
\]

\subsection{The case $n>6$} The examples above naturally lead us to the following generalization.
\begin{theorem}
\label{thm:zn-general}
Let $n\ge5$ and let  $g=\gcd(n,3)$.  
Define
$$
z_n\;=\begin{cases}
\frac{\displaystyle
      \Bigl(\prod_{i=5}^{\,n} \Delta_{12i}\Bigr)^{3}\;
      \bigl(\Delta_{134}\,\Delta_{234}\bigr)^{\,n-3}}
     {\displaystyle
      \bigl(\Delta_{123}\,\Delta_{124}\bigr)^{\,2n-9}}, \text{ if  } g=1,\\
			\frac{\displaystyle
      \Bigl(\prod_{i=5}^{\,n} \Delta_{12i}\Bigr)\;
      \bigl(\Delta_{134}\,\Delta_{234}\bigr)^{\,\frac{n}{3}-1}}
     {\displaystyle
      \bigl(\Delta_{123}\,\Delta_{124}\bigr)^{\,\frac{2n}{3}-3}},   \text{  if  } g=3\\
			\end{cases}
$$
Then \( z_n \) is a relative projective invariant, and its weight is equal
$$
\operatorname{wt}(z_n)=-\frac{1}{g}.
$$
\end{theorem}

	\begin{proof}
We will show that \( z_n \) is a relative invariant and compute its weight.  
According to the transformation formula for the determinant \( \Delta_{ijk} \), we have:
\[
\Delta_{ijk} \mapsto
\frac{(c_1x_i + c_2y_i + c_3)(c_1x_j + c_2y_j + c_3)(c_1x_k + c_2y_k + c_3)}{D}
\cdot \Delta_{ijk}.
\]

Each determinant in the numerator or denominator of \( z_n \) contributes the corresponding multiplicative factor to the transformation \( T(z_n) \). Let us count the total power of \( D \) in this rational function.

Assume first that \( g = 1 \). In the numerator, we have \( n - 4 \) determinants of the form \( \Delta_{12i} \), each raised to the third power, and two determinants \( \Delta_{134} \) and \( \Delta_{234} \), each raised to the power \( n - 3 \). In the denominator, we have two determinants raised to the power \( 2n - 9 \). Since each determinant contributes a factor of \( 1/D \) to \( T(z_n) \), we get:
\[
\operatorname{deg}_D(T(z_n)) = -3(n - 4) - 2(n - 3) + 2(2n - 9) = -n.
\]

Similarly, each factor \( (c_1x_i + c_2y_i + c_3) \) appears in total with exponent 3 for every \( i \), hence:
\[
T(z_n) = J(T)^{-1} \cdot z_n,
\]
and thus \( z_n \) is a relative invariant of weight \( -1 = -\tfrac{1}{g} \).

Now consider the case \( g = 3 \). A similar calculation yields:
\[
\operatorname{deg}_D(T(z_n)) = -(n - 4) - 2\left(\tfrac{n}{3} - 1\right) + 2\left(\tfrac{2n}{3} - 3\right) = -\tfrac{n}{3},
\]
and each factor \( (c_1x_i + c_2y_i + c_3) \) appears with exponent 1, so that
\[
T(z_n) = J(T)^{-1/3} \cdot z_n,
\]
which means \( z_n \) is a relative invariant of weight \( -\tfrac{1}{3} \).

Hence, in general, we have \( \operatorname{wt}(z_n) = -\tfrac{1}{g} \),
as claimed.
\end{proof}

We are now in a position to prove that \( \mathcal{I}_n \) is a simple extension of \( \mathcal{I}_{n,0} \), and that \( z_n \) is a primitive element of this extension.  
To do so, it suffices to show that \( \tfrac{1}{g} \) is the smallest possible weight (in absolute value) that can occur.

\begin{theorem}
Let
\[
W_n := \left\{ \omega \in \mathbb{Q} \,\middle|\, \exists\, F \in \mathcal{I}_n \text{ with } \operatorname{wt}(F) = \omega \right\}
\]
be the additive group of all possible weights of first-order relative projective invariants for \( n \) points, and let \( g = \gcd(n, 3) \). Then
\[
W_n = \tfrac{1}{g} \mathbb{Z},
\]
and hence \( \mathcal{I}_n = \mathcal{I}_{n,0}(z_n) \).
\end{theorem}

\begin{proof}

Consider the one-parameter subgroup of \( PGL(3, \mathbb{R}) \) acting by
\[
(x_i, y_i) \mapsto \left( \frac{x_i}{1 - tx_i}, \frac{y_i}{1 - tx_i} \right).
\]
The Jacobian of this transformation is
\[
(1 - tx_i)^{-3}.
\]

Since \(F\in\mathcal{I}_n\) is, in general, a rational function in the variables
\(x_i,y_i,p_i,q_i\) $i=1,\ldots, n$, write \(F=P/Q\) with coprime polynomials \(P,Q\).
Fix an index \(i\) and consider the substitution induced by the prolonged action 
\[
(x_i, y_i, p_i, q_i) \mapsto \left(
\frac{x_i}{1 - tx_i},\;
\frac{y_i}{1 - tx_i},\;
-(1 - tx_i)\left( t p_i x_i + t q_i y_i - p_i \right),\;
(1 - tx_i) q_i
\right),
\]
from Theorem~\ref{thm:first-prol}.
By factoring out the maximal power of \((1-tx_i)\) common to all terms, there exist integers
\(m_P(i),m_Q(i)\in\mathbb{Z}\) and polynomials \(\widehat P_i,\widehat Q_i\) not divisible by \((1-tx_i)\) such that
\[
P \mapsto(1-tx_i)^{m_P(i)}\widehat P_i,\qquad
Q\mapsto(1-tx_i)^{m_Q(i)}\widehat Q_i.
\]
Hence
\[
F \mapsto(1-tx_i)^{m_P(i)-m_Q(i)}\cdot \widehat F_i,\qquad
\widehat F_i:=\widehat P_i/\widehat Q_i,
\]
and therefore the exponent \(m_i:=m_P(i)-m_Q(i)\) is an integer.
On the other hand, by the definition of weight and the Jacobian factor \((1-tx_i)^{-3}\),
the relative invariant \(F\) transforms as
\(
F \mapsto (1-tx_i)^{-3\omega}\cdot F.
\)
Comparing the exponents of the factor \((1-tx_i)\) for this fixed \(i\) yields
\[
-3\omega=m_i\in\mathbb{Z}.
\]
Since the index \(i\) was arbitrary, we conclude that \(-3\omega\in\mathbb{Z}\).

Now consider another one-parameter subgroup given by
\[
(x_i, y_i) \mapsto (e^t x_i, y_i).
\]
The Jacobian of this transformation is \( e^t \), so the action on a relative invariant \( F \) of weight \( \omega \) is
\[
F \mapsto e^{nt\omega} \cdot F.
\]
According to Theorem~\ref{thm:first-prol}, the prolonged action of this transformation is
\[
(x_i, y_i, p_i, q_i) \mapsto \left(e^t x_i,\; y_i,\; \frac{p_i}{e^t},\; q_i\right).
\]

Factoring out the maximal power of \(e^t\) common to all terms, there exist integers \(m_P(i),m_Q(i)\in\mathbb{Z}\) and polynomials \(\widehat P_i,\widehat Q_i\) (not divisible by \(e^t\)) such that
\[
P \mapsto e^{t\,m_P(i)}\,\widehat P_i,\qquad
Q \mapsto e^{t\,m_Q(i)}\,\widehat Q_i.
\]
Therefore
\[
F \mapsto e^{t\,\bigl(m_P(i)-m_Q(i)\bigr)}\cdot \widehat F_i,\qquad
\widehat F_i:=\widehat P_i/\widehat Q_i,
\]
and the exponent \(m_i:=m_P(i)-m_Q(i)\) is an integer. Comparing this with the weight transformation
\(
F \mapsto e^{n t \omega}\cdot F
\)
yields
\[
n\omega = m_i \in \mathbb{Z}.
\]
Since \(i\) was arbitrary, we conclude that \(n\omega\in\mathbb{Z}\).

Combining these results, we find that
\[
\omega \in \tfrac{1}{3} \mathbb{Z} \cap \tfrac{1}{n} \mathbb{Z} = \tfrac{1}{g} \mathbb{Z}, \quad \text{where } g= \gcd(n, 3),
\]
so
\[
W_n \subseteq \tfrac{1}{g} \mathbb{Z}.
\]

Theorem~\ref{thm:zn-general} guarantees the existence of a relative invariant \( z_n \in \mathcal{I}_n \) of weight \( -\tfrac{1}{g} \), which implies
\[
-\tfrac{1}{g} \in W_n \quad \Rightarrow \quad \tfrac{1}{g} \mathbb{Z} \subseteq W_n.
\]
Thus, we obtain the equality:
\[
W_n = \tfrac{1}{g} \mathbb{Z}.
\]

Now consider an arbitrary relative invariant \( F \in \mathcal{I}_n \). From the above, its weight \( \operatorname{wt}(F) \) must be divisible by \( -\tfrac{1}{g} \), i.e., \( -g \cdot \operatorname{wt}(F) \in \mathbb{Z} \). Since \( z_n \) has weight \( -\tfrac{1}{g} \), the product
\[
F \cdot z_n^{-g \cdot \operatorname{wt}(F)}
\]
has weight zero, and hence belongs to the field of absolute invariants \( \mathcal{I}_{n,0} \). Therefore, we can write
\[
F = z_n^{-g \cdot \operatorname{wt}(F)} \cdot F'
\quad \text{for some } F' \in \mathcal{I}_{n,0}.
\]

This completes the proof that \( \mathcal{I}_n = \mathcal{I}_{n,0}(z_n) \).
\end{proof}

Note that in constructing the primitive elements, we may use the polynomials \( \delta_{ijk} \) instead of the determinants \( \Delta_{ijk} \), since their product \( \delta_{ijk} \cdot \Delta_{ijk} \) is an absolute invariant.

\section{Conclusions}
We have established a complete, minimal description of the field of first-order joint projective invariants under the diagonal $PGL(3,\mathbb R)$-action, proved that this field is obtained by adjoining a single weight  $-1$ relative invariant to the field of absolute invariants, and clarified the geometric r?le of each generator.  
Despite the fact that we have provided a complete algebraic description of all relative invariants of weight~$-1$, the explicit computation of the corresponding image projective invariants remains a nontrivial task. In particular, integrals of the form
\[
\underbrace{\int\!\cdots\!\int}_{2n \text{  integrals} } z_n\,dx_1\,dy_1\,\cdots\,dx_n\,dy_n,
\]
represent $2n$-fold projective invariants constructed from relative densities~$z_n$ evaluated at $n$ points. Such integrals, although projectively invariant, are generally intractable to compute directly for~$n > 3$, due to the rapidly growing dimensionality of the configuration space.

In future work, we aim to develop analytical and numerical methods for reducing the complexity of evaluating these high-order projective integrals. This creates an opportunity for practical applications in computer vision, particularly in image classification tasks involving projective transformations.


\begin{thebibliography}{99}

\bibitem{Hu}
Hu, M. K. (1962). Visual pattern recognition by moment invariants. \textit{IRE Transactions on Information Theory}, \textit{8}(2), 179--187.

\bibitem{SF}
Suk, T., \& Flusser, J. (2011). Affine moment invariants generated by graph method. \textit{Pattern Recognition}, \textit{44}(9), 2047--2056.

\bibitem{FSB}
Flusser, J., Suk, T., \& Zitova, B. (2017). \textit{2D and 3D Image Analysis by Moments}. John Wiley \& Sons.

\bibitem{Open}
Li, E., Mj, H., Xu, D., \& Li, H. (2019). Image projective invariants. \textit{IEEE Transactions on Pattern Analysis and Machine Intelligence}, \textit{41}, 1144--1157.

\bibitem{Van}
Van Gool, L., Moons, T., \& Oosterlinck, A. (1995). Vision and Lie’s approach to invariance. \textit{Image and Vision Computing}, \textit{13}(4), 259--277.

\bibitem{Suk2000}
Suk, T., \& Flusser, J. (2000). Point-based projective invariants. \textit{Pattern Recognition}, \textit{33}(2), 251--261.

\bibitem{Olver2023}
Olver, P. J. (2023). Projective invariants of images. \textit{European Journal of Applied Mathematics}, \textit{34}(5), 936--946.

\bibitem{Wang}
Wang, Y. B., Wang, X. W., Zhang, B., \& Wang, Y. (2015). Projective invariants of D-moments of 2D grayscale images. \textit{Journal of Mathematical Imaging and Vision}, \textit{51}(2), 248--259.

\bibitem{Olver1999}
Olver, P. J. (1999). Moving frames and joint differential invariants. \textit{Regular \& Chaotic Dynamics}, \textit{4}(4), 3--18.

\bibitem{Olver2001}
Olver, P. J. (2001). Joint invariant signatures. \textit{Foundations of Computational Mathematics}, \textit{1}(3), 255--271.

\bibitem{Olver2020}
Polat, G. G., \& Olver, P. J. (2020). Joint differential invariants of binary and ternary forms. \textit{Portugaliae Mathematica}, \textit{76}(2), 169--204.

\bibitem{Ros}
Rosenlicht, M. (1963). A remark on quotient spaces. \textit{Anais da Academia Brasileira de Ci?ncias}, \textit{35}(4), 487--489.

\bibitem{Sp}
Springer, T. A. (1977). \textit{Invariant Theory} (Vol. 585). Springer.

\end{thebibliography}
\end{document}